\documentclass[11pt,twoside,a4paper]{amsart}
\usepackage{amssymb}

\usepackage[english]{babel}
\usepackage{amsmath}
\usepackage{amsthm,amssymb,latexsym}
\usepackage[all,cmtip]{xy}
\usepackage{url,ae}
\usepackage{t1enc}
\usepackage{fancyheadings}
\usepackage{mathrsfs}
\usepackage{graphicx}
\usepackage{eufrak} 
\usepackage{geometry}

\date{\today}

\def\1{{\bf 1}}

\def\deg{\text{deg}\,}

\def\w{\wedge}

\def\dbar{\bar\partial}

\def\R{{\mathbb R}}
\def\C{{\mathbb C}}
\def\w{{\wedge}}

\def\S{{\mathcal S}}

\def\Hom{{\rm Hom\, }}

\def\rank{{\rm rank\, }}

\def\E{{\mathcal E}}

\def\O{{\mathcal O}}

\def\U{{\mathcal U}}

\def\J{{\mathcal J}}
\def\nbh{neighborhood }

\def\be{\begin{equation}}
\def\ee{\end{equation}}

\def\Ok{\mathcal O}

\def\V{{\mathcal V}}

\newtheorem{thm}{Theorem}[section]
\newtheorem{lma}[thm]{Lemma}

\newtheorem{prop}[thm]{Proposition}

\theoremstyle{definition}

\newtheorem{df}[thm]{Definition}

\theoremstyle{remark}

\newtheorem{preremark}[thm]{Remark}
\newtheorem{preex}[thm]{Example}

\newenvironment{remark}{\begin{preremark}}{\qed\end{preremark}}
\newenvironment{ex}{\begin{preex}}{\qed\end{preex}}

\numberwithin{equation}{section}

\usepackage{xcolor}

\title[]{The flatness of the $\Ok$-module of smooth functions and integral representation}

\begin{document}

\date{\today}

\author[Mats Andersson]{Mats Andersson}

\address{Department of Mathematics\\Chalmers University of Technology and the University of
Gothenburg\\S-412 96 G\"OTEBORG\\SWEDEN}

\email{matsa@chalmers.se}

\subjclass{}

\thanks{The author  was
  partially supported by the Swedish
  Research Council}

\begin{abstract}
We give a proof of the well-known fact that the $\Ok$-module $\E$ of smooth functions is flat
by means of residue theory and integral formulas.  A variant of the proof gives  
a related statement for
classes of functions of lower regularity. We also prove a Brian\c con-Skoda type theorem for
ideals of the form $\E a$, where $a$ is an ideal in $\Ok$. 
\end{abstract}


\maketitle

\section{Introduction}
Let $X$ be a complex manifold of dimension $n$ with structure sheaf $\Ok$ of holomorphic functions, 
and let $\E$ be the analytic sheaf of smooth functions. It is well-known, and first proved by 
Malgrange  already in the '60s,  \cite{Mal}, that $\E$ is a flat $\Ok$-module; that is, an exact sequence
of $\Ok$-modules remains exact when tensored by $\E$.  It is enough to prove,  see Section~\ref{flat}, 
noting that  $\E^{\oplus m}=\E\otimes_\Ok \Ok^{\oplus m}$:
If 
\begin{equation}\label{bas}
\Ok^{\oplus m_2}\stackrel{f_2}{\to}\Ok^{\oplus m_1}\stackrel{f_1}{\to}\Ok^{\oplus m_0}
\end{equation}
is exact, then the induced sequence
\begin{equation}\label{bas1}
\E^{\oplus m_2}\stackrel{f_2}{\to}\E^{\oplus m_1}\stackrel{f_1}{\to}\E^{\oplus m_0}
\end{equation}
is exact, that is, locally there is a smooth solution to $f_2 \psi=\phi$ for each smooth $\phi$ such that 
$f_1\phi=0$.  It is in fact enough to check the case when $m_0=1$. 



Our first aim of this note is to give a proof based on residue theory and an integral formula
that provides the desired smooth solution $\psi$.
The idea to use integral formulas to
find holomorphic solutions to this kind of equations, often referred to as division problems, 
was introduced by Berndtsson
in \cite{Be}. It was further developed and adapted for a variety of situtions, see, e.g.,  
\cite{Pa, BP, BY, BGVY,MH, Maz, RSW, AC, Aco, A3, AW1, ASS, AG}. 

%

\smallskip
One should notice that for instance the $\Ok$-module $C$ of continuous functions is not flat. 
Let us consider the simple exact
sequence 
\begin{equation}\label{enkel1}
\Ok\stackrel{f_2}{\to}\Ok^{\oplus 2}\stackrel{f_1}{\to}\Ok,
\end{equation}
in a \nbh of the origin in $\C^2_{x_1,x_2}$, where $f_1=(x_1,x_2)$ and $f_2=(-x_2, x_1)^t$.  
Observe that $C^{\oplus m}=C\otimes_\Ok \Ok^{\oplus m}$.
The induced sequence $C\to C^{\oplus2}\to C$ is not exact.  For instance,  
$\phi=(-x_2 \ x_1)^t/|x|^{1/3}$ is continuous and $f_1\phi=0$ but there is no continuous 
function $\psi$ such that $f_2\psi=\phi$;  in fact, since $f_2$ is pointwise injective outside the origin the
only possible solution is $\psi=1/|x|^{1/3}$.  However, if  $\phi$ is in $C^1$ and $f_1\phi=0$,
then there is indeed a continuous solution to $f_2\psi=\phi$, cf.~Example~\ref{plot} below.

\smallskip
Following \cite{AW1} one can associate a certain current $U$ with \eqref{bas}, see Section~\ref{posa}.

\begin{thm}\label{main}
Assume that \eqref{bas} is exact in a \nbh of $0\in \C^n$ 
and let $M$ be the order of the associated current $U$.
 
\smallskip

\noindent (i) If $\phi$ is in $\E^{\oplus m_1}$ and $f_1\phi=0$, then there is $\psi$ in
$\E^{\oplus m_{2}}$ such that $f_2\psi=\phi$. 

\smallskip
\noindent (ii) If  $\phi$ in $C^{k+2M+c_n}\otimes\Ok^{\oplus m_1}$ and 
$f_1\phi=0$, then there is  $\psi$ in $C^{k}\otimes\Ok^{\oplus m_{2}}$ such that $f_2\psi=\phi$.
Here $c_n$ is a constant that only depends on the dimension $n$.
 \end{thm}

In view of the discussion above,  (i) immediately implies that $\E$ is a flat $\Ok$-module. 

\smallskip
The proof of Theorem~\ref{main} relies on of the construction of  
weighted integral representation formulas in \cite{A3},  the
residue currents associated with free resolutions in \cite[Section~5]{AW1}, and
the special choice of weight in \cite{Aco}.

\smallskip

The classical  Brian\c con-Skoda
theorem, first proved in \cite{BS},  states that if 
$a\subset \Ok$ is an ideal and $\phi\in\Ok$ is a function
such that 
\begin{equation}\label{pottaska}
|\phi|\le C|a|^{\mu+r-1},
\end{equation}
 where
 \begin{equation}
\mu:=\min(m,n)
\end{equation}
and $m$ is the minimal number of generators for $a$,
then $\phi$ is in $a^r$.  Here $|a|=\sum|a_j|$  for a (finite) set $a_j$ of generators. 
Any other set of generators gives rise to a comparable quantity. 
By the Nullstellensatz $\phi$ is in the ideal $a^r$ if it vanishes to a large enough order
at the zero set of $a$. The important point is that the condition \eqref{pottaska} 
is uniform in both $a$ and $r$
(if we replace $\mu$ by $n$). 

There is a purely algebraic formulation in terms of integral closures,
see, e.g., \cite{BS}, and there are general versions for Noetherian,
even non-regular, rings, \cite{Hun}. 
We do not know if there is some kind of algebraic analogue for a non-Noetherian ring like $\E$. 
However, we have the following analytic variant for ideals $\E a\subset \E$,
where $a$ are ideals in $\Ok$.  Notice that $(\E a)^r=\E a^r$.

\begin{thm}\label{main3}
Assume that $\phi$ is a germ of a smooth function at $0$,
$a\subset \Ok$ is an ideal and $r$ is a positive integer. If there are constants $C_\alpha$ such that 
\begin{equation}\label{pontus}
|\partial^{\alpha}_{\bar z}\phi| \le C_\alpha |a|^{\mu+r-1}
\end{equation}
for all multi-indices $\alpha\ge 0$, then $\phi$ is in $\E a^r$.
\end{thm}

Notice that if $a_1,\ldots, a_m$ generate $a$ and $\phi$ is in $\E a^r$, then 
$$
\phi=\sum_{|I|=r} \xi_I a_1^{I_1}\cdots a_m^{I_m},
$$
where $\xi_I$ are in $\E$. Noting that $|a^r|\sim |a|^r$, thus 
$$
\dbar^\alpha_{\bar z}\phi=\sum_{|I|=r} (\dbar^\alpha_{\bar z}\xi_I) a_1^{I_1}\cdots a_m^{I_m}.
$$
Therefore,  \eqref{pontus} holds for all $\alpha$ with $\mu=1$, and thus a condition like \eqref{pontus}
is necessary.

\smallskip 

The plan of this note is as follows.  In Section~\ref{prel} we recall  some results about
residue theory and integral representation of solutions to division problems. In the last sections
we provide the proofs of Theorems~\ref{main} and \ref{main3}.

\section{Preliminaries}\label{prel}
We first recall some material that is basically known, but presented in a way
that is adapted for the proofs in the last two sections. 

\subsection{Flatness}\label{flat}
Let $R$ be a commutative ring and $M$ an $R$-module. There are many equivalent statements
with the meaning that $M$ is flat.  One is the following:

\noindent {\it  $M$ is flat if and only if $J\otimes_R M\to R\otimes_R M= M$
is injective for each finitely generated ideal $J\subset R$.}

It is quite easy to see that this holds if and only if for any finite relation $\sum_1^d r_j \phi_j=0$
there are a finite $R$-matrix $A$ and a tuple $b$ of elements in $M$ such that $\phi=A b$
and $(r_1\ \ldots r_d) A=0$.  
 
\smallskip
In our case $R$ is the local Noetherian ring $\Ok$ at some point,  and so all ideals are finitely generated.  
By Oka's lemma there is a holomorphic matrix $f_2$ such that
\begin{equation}\label{potta}
R^{m_2}\stackrel{f_2}{\longrightarrow} R^d \stackrel{(r_1\ \ldots r_d)}{\longrightarrow} R
\end{equation}
is exact (which means that  $(r_1\ \ldots r_d) f_2=0$) in a \nbh of $0$.   
Thus we are precisely in the situation in the
introduction of this note, with $m_0=1$ and $m_1=d$, and so the flatness of the $\Ok$-module $\E$
follows from the exactness of \eqref{bas1}.


\subsection{Integral representation}\label{skutt1}
We first describe the idea in \cite{A1} to construct representation formulas for holomorphic functions in an
open set $\Omega\subset\C^n$.  Let $z$ be a fixed point in $\Omega$, let $\delta_{z-\zeta}$
be interior multiplication by the vector field
$$
2\pi i \sum_1^n (\zeta_j-z_j)\frac{\partial}{\partial \zeta_j},
$$
and let $\nabla_{\zeta-z}=\delta_{\zeta-z}-\dbar$. 
Notice that $\nabla_{\zeta-z}$ satisfies Leibniz' rule  
\begin{equation}\label{lei}
\nabla_{\zeta-z}(\alpha\w\beta)=\nabla_{\zeta-z}\alpha\w \beta+(-1)^{\deg\alpha}\alpha\w \nabla_{\zeta-z}\beta.
\end{equation}
We say that a current 
$g=g_{0,0}+g_{1,1}+\ldots+g_{n,n}$, where lower indices denote bidegree, is a {\it weight}
with respect to $z$ if $\nabla_{\zeta-z}g=0$,
$g$ is smooth in a \nbh of $z$  and $g_{0,0}(z)=1$.  
If $g^1$ and $g^2$ are weights, 
and one of them is smooth, then by \eqref{lei} also $g^1\w g^2$ is a weight.

\smallskip
Let $b=\partial_\zeta|\zeta-z|^2/2\pi i |\zeta-z|^2$
and for $\zeta\neq z$ consider the form
\begin{equation}\label{kaftan}
v^z=\frac{b}{\nabla_{\zeta-z}b}=b+ b\wedge \dbar_\zeta b+\cdots +
b\wedge(\dbar_\zeta b)^{n-1};
\end{equation}
in \cite{A1} it is called the full Bochner-Martinelli form because the component $v^z_{n,n-1}$ of bidegree
$(n,n-1)$ is precisely the classical Bochner-Martinelli kernel with pole at the point  $z$. 
It is easy to see that $\nabla_{\zeta-z}v^z=1$.
Since $v^z$ is locally integrable it is has a natural current extension across $z$ and 
\begin{equation}\label{skata}
\nabla_{\zeta-z} v^z=1-[z],
\end{equation}
where $[z]$ denotes the Dirac measure at $z$ considered as an $(n,n)$-current,
cf.~\cite[Section~2]{A1}.

\begin{prop}\label{vikt}
Assume that $g$ is a weight with respect to $z\in \Omega$ with compact support
in $\Omega$.  If $\Phi=\Phi_{0,0}+\cdots +\Phi_{n,n}$ is a smooth form in $\Omega$ and
$\nabla_{\zeta-z}\Phi=0$, then
\begin{equation}\label{periskop}
\Phi_{0,0}(z)=\int g\w \Phi=\int (g\w \Phi)_{n,n}.
\end{equation}
\end{prop}

In particular the formula holds for $\Phi= \phi$ if $\phi$ is holomorphic in $\Omega$. 

\begin{proof}
Notice that $v^z\w g\w \Phi$ is a well-defined current. By \eqref{lei} and \eqref{skata},
$$
\nabla_{\zeta-z}(v^z\w g\w \Phi)=\nabla_{\zeta-z}v\w g\w \Phi=(1-[z])\w g\w \Phi=
g\w \Phi-   \Phi_{0,0}[z].
$$
For degree reasons it follows that 
$$
d((v^z\w g\w \Phi)_{n,n-1}=(g\w \Phi)_{n,n}-   \Phi_{0,0}[z]
$$
and so \eqref{periskop} follows from Stokes' theorem.
\end{proof}

Assume that  $E\to \Omega$ is a holomorphic vector bundle, 
$\Phi$ takes values in $E$, $g$ takes values in $\Hom(E_\zeta,E_z)$,
$g$ is smooth in a \nbh of $z$ 
and $g_{0,0}(z)=I_{E_z}$. Then the same proof gives \eqref{periskop} 
for these  $g$ and $\Phi$.



\begin{ex}\label{totta}
Assume that $\Omega$ is the unit ball. If
$$
s=\frac{1}{2\pi i} \frac{\dbar|\zeta|^2}{|\zeta|^2-z\cdot\bar\zeta},
$$
then $\delta_{\zeta-z}s=1$ when $\zeta\neq z$. If  
$$
u=\frac{s}{\nabla_{\zeta-z}}=s\w\big( 1+\dbar s+\cdots +(\dbar s)^{n-1}\big)
$$
then $\nabla_{\zeta-z}u=1$ when $\zeta\neq z$.  If $\chi$ is a cutoff function in $\Omega$ that is $1$ in a \nbh
of the closure of a smaller ball $\Omega'$, then  
$$
g=\chi -\dbar\chi\w u
$$
is a weight with respect to $z$ for each $z\in \Omega'$, depends holomorphically on
$z$, and has compact support in $\Omega$.
\end{ex}

\subsection{Residues associated with  generically exact Hermitian complexes.}\label{posa}
Assume that we have a complex 
\begin{equation}\label{bastu2}
0\to E_N\stackrel{f_N}{\longrightarrow}E_{{N-1}}
\stackrel{f_{N-1}}{\longrightarrow}\cdots \stackrel{f_3}{\longrightarrow}
E_{2}\stackrel{f_2} {\longrightarrow}E_{1}
\stackrel{f_1}{\longrightarrow}E_0
\end{equation}
of Hermitian vector bundles over  a complex manifold $X$ that it is pointwise exact
in $X\setminus Z$, where $Z$ is an analytic set of positive codimension. 
Let $E=\oplus E_j$ and $f=\oplus f_j$, and let $\nabla_f=f-\dbar$.
In \cite{AW1} were defined currents $U$ and  $R$ with components $U^\ell_k$ and $R^\ell_k$  
with values in $\Hom(E_\ell,E_k)$
and of bidegree $(0,\ell-k-1)$ and $(0,\ell-k)$, respectively, such that 
\begin{equation}\label{nablau}
\nabla_f\circ U + U\circ\nabla_f=I_E-R
\end{equation}
and $R$ has support on $Z$.  To be precise, one introduces a superstructure on $E\oplus T^*\Omega$
so that $U$, $f$ and $\nabla_f$ are odd mappings whereas  $R$ is even; for details
see \cite[Section~2]{AW1}. However, for the purpose of this paper, the precise signs are not essential. 

It is proved that if $\phi$ is a holomorphic section of $E_\ell$ such that $f_\ell\phi=0$ and, in addition,
$R^\ell\phi=0$, then locally there is a holomorphic solution to $f_{\ell+1}\psi=\phi$.

\begin{ex}\label{koszul1}
Let $A\to X$ be a Hermitian vector bundle of rank $m$ and let $a$ be a holomorphic  section of its dual.
If $E_k=\wedge^k A$ we get the so-called Koszul complex by letting
$f_k$ be interior multiplication $\delta_a$ by $a$. If $Z$ is the set where $a$ vanishes, then the complex
is exact in $\Omega\setminus Z$.   If $\sigma$ is the section of $E$ over $X\setminus Z$ with pointwise
minimal norm such that $\delta_a\sigma=1$,  then 
$$
U_k^\ell=\sigma\w (\dbar\sigma)^{k-\ell}
$$
there. 
Assume now that $a=a_0 a'$, where $a_0$ is a section of a line bundle $L^*\to X$ 
and $a'$ is a non-vanishing section of $A^*\otimes L$.  Then $\sigma=\sigma'/a_0$, where
$\sigma'$ is smooth. Since $\sigma'\w \sigma'=0$,  thus 
$$
U_k^\ell=\frac{1}{a_0^{k-\ell}}\sigma'\w(\dbar\sigma')^{k-\ell-1}.
$$
It turns out that $U_k^\ell$ have extensions across $Z$ as principal value currents. 
Since $R^0_k=\delta_a U^0_{k+1}-\dbar U^0_k$ one can check  that 
$$ 
R_k^0=\dbar\frac{1}{a_0^{k}}\w \sigma'\w(\dbar\sigma')^{k-1}.
$$ 
For degree reasons this current must vanish when $k>n$ and since
$\rank A=m$ it must vanish if $k>m$. Using that
$a_0\dbar(1/a_0^{k+1})=\dbar(1/a_0^k)$ we conclude that 
\begin{equation}\label{struts1}
R^0=
\dbar\frac{1}{a_0^{\mu}}\w\omega,
\end{equation}
where $\omega$ is a smooth form.
\end{ex}

\begin{ex}\label{koszul2}
Let $A\to X$ and $a$ be as in the previous example. 
If $r\ge 2$ we can construct a similar complex such that $E_0=\C$, $E_1=\otimes_1^r A$ and
$f_1$ is $\delta_a\otimes\cdots \otimes\delta_a$. For a description of the whole complex, see,
e.g., \cite[Proof of Theorem~1.3]{A4}.  
If now $a=a_0a'$, it follows by considerations as in the previous example that
there is a smooth form $\omega$ such that 
 \begin{equation}\label{struts2}
R^0=\dbar\frac{1}{a_0^{\mu+r-1}}\w\omega.
\end{equation}
\end{ex}

\subsection{Free resolutions}\label{bagare}

Assume that we have a sequence \eqref{bas} in a \nbh  of a point $0\in \U\subset\C^n$. 
Let $\J$ be the ideal generated by the entries in $f_1$, or equivalently, the image of the mapping
$f_1$ in $\Ok$.  
After possibly shrinking $\U$ we can extend to an exact sequence  of sheaves
\begin{equation}\label{bastu1}
0\to \Ok^{\oplus m_N}\stackrel{f_n}{\longrightarrow}\Ok^{\oplus m_{n-1}}
\stackrel{f_{n-1}}{\longrightarrow}\cdots \stackrel{f_2}{\longrightarrow}
\Ok^{\oplus m_1}\stackrel{f_1}{\longrightarrow}\Ok^{\oplus m_0}\to \Ok^{\oplus m_0}/\J\to 0,
  \end{equation}
that is, a free resolution of the $\Ok$-module $\Ok^{m_0}/\J$ in $\U$. 
By Hilbert's zyzygy theorem we can assume that  $N\le n$.
We have an induced complex of (trivial) vector bundles
\eqref{bastu2},  where $\rank E_k=m_k$, which is pointwise exact
outside the zero set $Z$ of $f_1$ (if $m_0\ge 2$,  $Z$ is the set where $f_1$ does not have
optimal rank.) Let us equip these vector bundles with any Hermitian metrics, for instance
trivial metrics with respect to global frames, and let $R$ and $U$ be the associated
currents.  
A main result in \cite{AW1} is that the exactness of \eqref{bastu1} implies that 
$R^\ell=0$ for $\ell\ge 1$. This in turn implies that if $\phi\in \Ok(E_0)$
(and $\phi$ is pointwise generically in the image of $f_1$),  then
$f_1\psi=\phi$ has  locally a holomorphic solution if and only if $R\phi=0$.  

\begin{df} Given an exact sequence \eqref{bas} we let $M$ be the minimal order of such an associated current
$U$. 
\end{df}


\subsection{Hefer mappings}
Let us recall the idea from \cite{A3} to construct division-interpolation formulas.
Given a generically exact complex like \eqref{bastu2} 
in $\Omega\subset \C^n$ and $z\in\Omega$ 
one can locally find a tuple $H=(H^\ell_k)$, where 
$H^\ell_k$ is a smooth form-valued section of 
$\Hom(E_{k,\zeta},E_{\ell,z})$,  such that $H_k^\ell=0$ for $k<\ell$, 
$H_\ell^\ell=I_{E_\ell}$ when $\zeta=z$, and in general
\begin{equation}\label{heferlikhet}
\nabla_{\zeta-z}H_k^\ell=H^\ell_{k-1}f_k - f_{\ell+1}(z)H^{\ell+1}_k.
\end{equation}
In fact, if $\Omega$ is pseudoconvex, then one can find such an $H$ that depends
holomorphically on both $\zeta\in\Omega$ and $z\in\Omega'\subset\subset\Omega$.
In that case $H^\ell_k$ has bidegree $(k-\ell,0)$. 
It was proved in \cite{A3}, see also the proof of Proposition~4.2 in \cite{AG},
that if $U$ and $R$ are the currents above associated with the Hermitian
complex \eqref{bastu2}, then 
\begin{equation}\label{pudel}
f(z)HU+HUf+HR
\end{equation}
is $\nabla_{\zeta-z}$-closed. Let us fix a non-negative integer $\ell$ and let  
$g'$ be the component of \eqref{pudel} with values in $E_\ell$, that is, 
$$
g'=f_{\ell+1}(z)H^{\ell+1} U^\ell +H^\ell U^{\ell-1} f_\ell + H^\ell R^\ell=\sum_k f_{\ell+1}(z)H^{\ell+1}_k U^\ell_k+\sum_k H^\ell_k U^{\ell-1}f_\ell 
+\sum_k H^\ell_k R^\ell_k.
$$
Then $g'_{0,0}(z)= I_{E_{\ell,z}}$.   If $z\notin Z$, then  $g'(\zeta)$ is smooth in a \nbh of $z$,
and hence it is a weight.  
If $g$ is a
weight with compact support and $\Phi$ is as in Proposition~\ref{periskop}, but taking values in
$E_\ell$, then, since $g'\w g$ is a weight as well, 
\begin{equation}\label{periskop2}
\Phi_{0,0}(z)=\int (g'\w g\w \Phi)_{n,n}=\int (g' \Phi\w g)_{n,n}.
\end{equation}
Since \eqref{periskop2} holds for $z\in \Omega'\setminus Z$, and 
both sides are smooth in $z$, we conclude that \eqref{periskop2}
holds for all $z\in\Omega'$.  
%
%
We can write \eqref{periskop2} as
 \begin{equation}\label{pokemon}
\Phi_{0,0}(z)= f_{\ell+1}(z) \int_\zeta H^{\ell+1} U^\ell \Phi\w g +\int_\zeta H^\ell U^{\ell-1} f_\ell\Phi\w g + 
\int_\zeta H^\ell R^\ell\Phi \w g.
\end{equation}
If the sheaf complex \eqref{bastu1} is exact and $\ell\ge 1$, as mentioned above, then $R^\ell=0$, and hence 
\begin{equation}\label{pokemon2}
\Phi_{0,0}(z)=f_{\ell+1}(z)\int_\zeta H^{\ell+1} U^\ell \Phi\w g +\int_\zeta H^{\ell} U^{\ell-1} f_\ell\Phi\w g.
\end{equation}

In particular, if $\Phi=\phi$ is a holomorphic section of $E_\ell$ and $f_\ell\phi=0$, then
the middle term is a holomorphic solution to $f_{\ell+1}\psi=\phi$.

\section{Proof of Theorem~\ref{main}}
 
As we have seen the 
flatness of the $\Ok$-module $\E$ follows from Theorem~\ref{main}~(i).
%
%
In \cite{Mal} it is proved that $\E$ is a flat $C^\omega$-module, where $C^\omega$ is the sheaf of
real-analytic functions in $\R^N$. This result is obtained by a sophisticated study of local properties of 
real-analytic functions and goes via formal power series. The same proof can be applied to $\Ok$ instead of
$C^\omega$ and then gives the flatness of $\E$ as an $\Ok$-module. One can also derive this statement
quite easily directly from the flatness of $\E$ as a $C^\omega$-module without reference to the proof in
\cite{Mal}.  The proof of the flatness in this paper is independent of \cite{Mal} but 
relies on the possibility to resolve singularities, i.e., Hironaka's theorem, 
which we need to define the currents we use.

\subsection{Proof of Theorem~\ref{main}~(i)}
Let  $\Omega\subset\C^n$ be a \nbh of $0$ where we have the exact complex \eqref{bas}. In
a possibly smaller \nbh we can extend to an exact complex \eqref{bastu1} that we equip with
some Hermitian metrics. Let \eqref{bastu2} be the associated generically exact Hermitian complex
and let $U$ and $R$ be the associated currents. 
Let us identify $\Omega$ with the set
$ \{(\zeta,\bar\zeta)\in\C^{2n};\ \zeta\in \Omega\}$
and let  $\widetilde \Omega$ be an open neighborhood of $\Omega$  in $\C^{2n}_{\zeta,\omega}$.
Notice that \eqref{bastu2} induced a generically exact complex in $\widetilde\Omega$
if we let $\tilde f_\ell(\zeta,\eta):=f_\ell(\zeta)$.  The associated currents $\widetilde U$ and $\widetilde R$ 
are the tensor products $U\otimes 1$ and $R\otimes 1$. In particular, $\widetilde R^\ell=0$ for $\ell\ge 1$,
cf.~Section~\ref{bagare}. 
We will use formula \eqref{pokemon2} in $\widetilde\Omega$. 

If $\phi$ is a smooth section of $E_\ell$ in $\Omega$, then 
let us consider  the formal sum  
\begin{equation}\label{deff}
\tilde\phi(\zeta,\omega)=
\sum_{\alpha\ge 0} (\partial^\alpha_{\bar\zeta}\phi)(\zeta)
\frac{(\omega-\bar\zeta)^{\alpha}}{\alpha!}\chi(\lambda_{|\alpha|}(\omega-\bar\zeta)),
\end{equation}
where $\chi$ is a cutoff function in $\C^n$ which is $1$ in a neighborhood
of $0$, and $\lambda_k$ are positive numbers. 
If 
$\lambda_k\to\infty$ fast enough, then, possibly after shrinking $\Omega$,  the series \eqref{deff} converges to a 
smooth section of $E_\ell$ in $\widetilde \Omega$   such that 
\begin{equation}\label{rott1}
\tilde\phi(\zeta,\bar\zeta)=\phi(\zeta),
\end{equation}
and
\begin{equation}\label{rott2}
\dbar\tilde\phi(\zeta,\omega)=\O(|\omega-\bar\zeta|^{\infty}).
\end{equation}
Such a $\tilde\phi(\zeta,\omega)$, satisfying \eqref{rott1} and \eqref{rott2},
is called an almost holomorphic extension of $\phi$ from
$\Omega$ to $\widetilde\Omega$. 

If $\phi$ is real-analytic one can take $\lambda_k=1$ for all $k$;
then $\tilde\phi$ is the holomorphic extension of $\phi$.  The requirement on
$\lambda_k$ is related to which ultra-differentiable class $\phi$ belongs to, that is,
how fast its Fourier transform decays. If $\phi$ is in a certain such class
and $h$ is holomorphic, then $h\phi$ is in the same class.  

\begin{lma}\label{bongo}
Let $\phi$ be a smooth section of $E_\ell$ in $\Omega$,  let 
$v^z$ denote the Bochner-Martinelli form in $\tilde \Omega$
with respect to the point $(z,\bar z)$,  and let
$$
\Phi^z(\zeta,\omega)=\tilde\phi(\zeta,\omega)-\dbar\tilde\phi\wedge v^z.
$$
Then
$\Phi^z$ is smooth in $\zeta,\omega,z$
and $\nabla_{(\zeta,\omega)-(z,\bar z)}\Phi^z=0$.
Moreover, if $f_\ell \phi=0$, then
$f_\ell \Phi^z=0$.
\end{lma}

\begin{proof}
Since
$$
v^z=\frac{b}{\nabla_{(\zeta,\omega)-(z,\bar z)} b},
$$
where
$$
b=\frac{1}{2\pi i}\Big( \sum_1^n(\zeta_j-z_j)d\zeta_j+\sum_1^n(\omega_j-\bar z_j)d\omega_j\Big),
$$
cf.~\eqref{kaftan}, we have that
$$
\Phi^z(\zeta,\omega)=\tilde\phi(\zeta,\omega)+\sum_{\ell=1}^{2n}\frac{\O(|\omega-\bar\zeta|^\infty)}
{(|\zeta-z|^2+|\omega-\bar z|^2)^{\ell-1/2}},
$$
and thus $\Phi^z$  is smooth.
Since $\nabla_{(\zeta,\omega)-(z,\bar z)}v^z=1$ outside the point $(z,\bar z)$ it follows that
$\nabla_{(\zeta,\omega)-(z,\bar z)}\Phi^z=0$.
If  $f_\ell \phi(\zeta)=0$, then  
 $f_\ell (\partial^\alpha_{\bar\zeta}\phi)(\zeta)=0$
for all $\alpha$ 
and therefore
$f_\ell \tilde\phi(\zeta,\omega)=0$.
Thus also  
$f_\ell (\dbar\tilde\phi)(\zeta,\omega)=0$, and so the lemma follows.
\end{proof}

Let $\Omega'\subset\subset\Omega$ be an open subset and for each $z\in \Omega'$,
let $g$ be a smooth weight with respect to $(z,\bar z)\in\widetilde\Omega'$ with compact support
in $\widetilde\Omega$, cf.~Example~\ref{totta}; of course we may assume that $\widetilde\Omega$ 
is a ball.  

Assume that $\phi$ is a smooth section of $E_\ell$ in $\Omega$ and choose $\lambda_k$ such that \eqref{deff}
defines an almost holomorphic extension. Then also $f_\ell\phi$ admits such an extension,
in fact $\widetilde{f_\ell\phi}=\tilde f_\ell\tilde\phi$. 
We can then define the operator 
$$
T_\ell\colon \E(\Omega,R_\ell)\to \E(\Omega',E_{\ell+1});\quad 
T_\ell\phi(z)=\int_{\zeta,\omega} \tilde H \tilde U \Phi^z \w g, \quad z\in \Omega'.
$$
It follows from Lemma~\ref{bongo} that $T_\ell\phi$ is smooth in $\Omega'$. Moreover, if
$\ell\ge 1$ we have 
from \eqref{pokemon2} that
\begin{equation}\label{pokemon3}
\phi=f_{\ell+1}T_\ell\phi+ T_{\ell-1} (f_{\ell}\phi) 
\end{equation}
in $\Omega'$. Now Theorem~\ref{main}~(i) follows if $\ell=1$.

 \begin{remark}
If we instead use \eqref{pokemon} for $\ell=0$ 
we can obtain a proof
of the following residue condition for membership in an ideal of germs of smooth
functions in $\E a$, where $a\subset \Ok$.
 
\begin{thm}\label{main2}
Assume that $a\subset\Ok$ is an ideal. 
If $R$ is the residue current associated with the resolution \eqref{bastu1}  of $\Ok/\J$, then  
a germ of a smooth
function $\phi$ is in $\E a$ if and only if 
\begin{equation}\label{koks}
(\partial^\alpha_{\bar z}\phi) R=0
\end{equation}
for each multiindex $\alpha\ge 0$.
\end{thm}

This result was proved in \cite{Aco} in case $a$ is a complete intersection, and for a general ideal,
in fact for any submodule $a\subset \Ok^{\oplus m}$, in \cite{AW1}.
\end{remark}

\subsection{Proof of Theorem~\ref{main}~(ii)}

Let us first consider a simple example with lower regularity.  
\begin{ex}\label{plot}
Consider the sequence \eqref{enkel1} and assume that $\phi=(\phi_1,\phi_2)$ is in $C^1$ 
and that $f_1\phi=0$. This  means that
\begin{equation}\label{kamera}
x_1\phi_1+x_2\phi_2=0
\end{equation}
and thus  $\phi_2=0$ when $x_1=0$. Since $\phi_1$ is $C^1$ therefore
$\phi_2=x_1\psi$, where $\psi$ is continuous. In the same way $\phi_1=-x_2\tilde\psi$.
From \eqref{kamera} it follows that $\tilde\psi=\psi$ and hence $\psi$ is 
a continuous solution to $f_2\psi=\phi$.
\end{ex}

We will use the same set-up and notation as in the proof above of part (i)
except for the extension $\tilde \phi$ of $\phi$. 
Let $c_n$ be a positive integer. 
If $\phi$ is in $C^{c_n+2M+k}(\Omega,E_\ell)$,  
then  
\begin{equation}\label{defII}
\tilde\phi(\zeta,\omega):=
\sum_{|\alpha|\le c_n+M+k} (\partial^\alpha_{\bar\zeta}\phi)(\zeta)
\frac{(\omega-\bar\zeta)^{\alpha}}{\alpha!},
\end{equation}
is in $C^M(\widetilde\Omega,E_\ell)$,   
$\tilde\phi(\zeta,\bar\zeta)=\phi(\zeta)$,
and  
\begin{equation}\label{kanin}
\dbar\tilde\phi(\zeta,\omega)=\Ok\big(|\omega-\bar\zeta|^{c_n+M+k}\big).
\end{equation}
We have the following analogue to Lemma~\ref{bongo}.

\begin{lma}
There is a constant $c_n$, only depending on the dimension $n$, such that 
 if $\phi$ is in $C^{c_n+2M+k}( \Omega,E_\ell)$ and 
 $$
 \Phi^z= \tilde\phi(\zeta,\omega)-\dbar\tilde\phi\wedge v^z,
 $$
then 
 $\Phi^z$ is in $C^M( \widetilde\Omega,E_\ell)$ 
even after taking up to  $k$ derivatives with respect to $z$, and
$\nabla_{(\zeta,\omega)-(z,\bar z)}\Phi^z=0$.
Moreover, if $f_\ell \phi=0$, then
$f_\ell \Phi^z=0$.
\end{lma}

\begin{proof}
Notice that 
 $$
\Phi^z(\zeta,\omega)=\tilde\phi(\zeta,\omega)+\sum_{\ell=1}^{2n}\frac{\O(|\omega-\bar\zeta|^{c_n+M+k})}
{(|\zeta-z|^2+|\omega-\bar z|^2)^{\ell-1/2}}.
$$
If $c_n$ is suitably chosen, then we can take $k$ derivatives
with respect to $z$ and still remain in $C^M(\widetilde\Omega,E_\ell)$.
Since before it follows that $\nabla_{(\zeta,\omega)-(z,\bar z)}\Phi^z=0$.
If  $f_\ell \phi(\zeta)=0$, we have that 
 $f_\ell (\partial^\alpha_{\bar\zeta}\phi)(\zeta)=0$
for all $\alpha$ such that $|\alpha|\le M+c_n+k$,
and therefore
$f_\ell \tilde\phi(\zeta,\omega)=0$.
It follows that also 
$f_\ell (\dbar\tilde\phi)(\zeta,\omega)=0$.
\end{proof}

We can now conclude the proof of Theorem~\ref{main}~(ii).  
Notice that Proposition~\ref{vikt} holds if $g$ has order $M$
(and is smooth at $z$) and $\Phi$ is at least in $C^M$. 
Our weight $g$ leading to \eqref{pokemon} 
contains the current $R$, that may have order $M+1$, and to avoid keeping track
relevant components, although it is not necessary, it is convenient to replace
$c_n$ by $c_n+2$.  We can then proceed precisely as in the smooth case
and obtain integral operators 
$$
T_\ell\colon C^{c_n+2M+k}(\Omega,E_\ell)\to C^k(\Omega',E_{\ell-1})
$$
such that \eqref{pokemon3} holds  in $\Omega'$  if $\ell\ge 1$.
Now part (ii) of Theorem~\ref{main} follows if $\ell=1$.

\section{Proof of Theorem~\ref{main3}}

Assume that we have a generically exact Hermitian complex \eqref{bastu2} in $X$
with associated currents $U$ and $R$.
If $\pi\colon X'\to X$ is a modification,  then the pullback of \eqref{bastu2} is a generically
exact Hermitian complex in $X'$ and thus we have associated currents $U'$ and $R'$.
It follows from the definition, cf.~\cite{AW1}, that 
$\pi_*U'=U$ and $\pi_*R'=R$.  In particular, if we have a Koszul complex on $X$
generated by the section $a$ of $A^*$ as in Example\ \ref{koszul1}, then the pull-back
is a Koszul complex in $X'$ generated by the section $\pi^*a$ of $\pi^*A^*$. 
It follows from the example that if $\pi^*a=a_0 a'$, where $a'$ is non-vanishing, then
$$
R^0=\pi_*\big(\frac{1}{a_0^\mu}\w \omega\big),
$$
where $\omega$ is a smooth form in $X'$.  
 In the same way if $a\subset\Ok$ is an ideal in $X$ and we take the Koszul-type complex
 in Example~\ref{koszul2} associated with $a^r$, then there is a smooth form on $X'$ such that 
the associated current has the form
\begin{equation}\label{skrov}
 R^{0}=\pi_*\big(\frac{1}{a_0^{\mu+r-1}}\w \omega\big).
\end{equation}

Let us now assume that $X=\Omega$ is a \nbh of $0\in\C^n$
and let $\phi$ be a smooth function in $\Omega$.
If we proceed precisely as in the proof of Theorem~\ref{bas}(i) above but with the currents
$U$ and $R$ from Example~\ref{koszul2}, we get from \eqref{pokemon2} for $\ell=0$
the formula 
\begin{equation}\label{kloak}
\phi=f_1 T_1\phi +S\phi, 
\end{equation}
where
\begin{equation}\label{kloak2}
S\phi(z)=\int  \tilde H\tilde R^0\Phi^z\w g.
\end{equation}
In view of \eqref{deff} we get from \eqref{kloak} and \eqref{kloak2},
cf.~\cite[Theorem~1.1]{Aco}:
 
\begin{lma} \label{kork}
If  
\begin{equation} \label{korall}
(\partial_{\bar z}^\alpha\phi) R^{0}=0
\end{equation}
for all $\alpha\ge 0$,   then $\tilde R^0\Phi^z=0$ and 
$\phi$ is in $\E a^r$.
\end{lma}

\begin{proof}[Proof of Theorem~\ref{main3}]
In view of Lemma~\ref{kork}
have to prove that 
\eqref{pontus} implies \eqref{korall}. 
Let as choose a modification $\pi\colon \Omega'\to \Omega$ as above.
For simplicity we skip the upper index $0$ and 
and write \eqref{skrov} as $R=\pi_* R'$,
where 
 \begin{equation}\label{pompa}
R'=\dbar\frac{1}{a_0^{\mu+r-1}}\w \omega
\end{equation}
and $\omega$ is a smooth form with values in $L^{\mu+r-1}$, where $L$ is the line bundle
defined by $a_0$. 
Let $\xi=\partial_{\bar z}^\alpha\phi$ for some $\alpha$. We must prove that $\xi R=0$.
Since $\xi$ is smooth, 
$$
\xi R=\pi_*(\pi^*\xi\cdot R').
$$
It is thus enough to verify that  $\pi^*\xi R'=0$.

Since $|a'|>0$, the assumption \eqref{pontus} implies that 
\begin{equation}\label{kokett}
|\pi^*\xi|\le C |a_0|^{\mu+r-1}
\end{equation}
for some constant $C$. 
In a \nbh $\V$ of a regular point on the zero set of $a_0$ we can assume that we have local coordinates
$s_1,\ldots, s_n$ such that $s_1^t=a_0$ for some positive integer $t$, or more correctly, $s_1^t$ is 
the representative of $a_0$ with respect to a local frame for $L$. In any case, in view of \eqref{kokett},
a Taylor expansion gives that $\pi^*\xi$ is a finite sum of terms 
$s_1^{t(\mu+r-1)-j}\bar s_1^j \nu$, where $\nu$ is smooth.  It is well-known that each such term
annihilates the residue current $\dbar(1/s_1^{t(\mu+r-1)})$.  In fact, already one single
factor $\bar s_1$ is enough. The "worst" case is thus $s_1^{t(\mu+r-1)}\nu$, and this factor
precisely annihilates $\dbar(1/s_1^{t(\mu+r-1)})$.
It follows that $\mu=\pi^*\xi \dbar(1/a_0^{\mu+r-1})$ has support on a set with codimension $\ge 2$. 
Since $\mu$ is pseudomeromorphic and has  bidegree $(0,1)$ it must therefore vanish identically in view of the 
dimension principle, see, e.g., \cite{ASS}. 

Alternatively we can assume that the modification is chosen so that 
$a_0$ locally is a monomial (normal crossings), say $a_0=s_1^{t_1}\cdots s_k^{t_k}$. 
Then \eqref{kokett} and a Taylor expansion reveals that $\pi^*\xi$ is a sum of terms
with a factor $\Ok(|s_1|^{t_1(\mu+r-1)}\cdots |s_k|^{t_k(\mu+r-1)})$, and each such term
annihilates $\dbar(1/ s_1^{t_1(\mu+r-1)}\cdots s_k^{t_k(\mu+r-1)})$.
\end{proof}

\def\listing#1#2#3{{\sc #1}:\ {\it #2},\ #3.}

\end{document}